\documentclass{amsart}
\usepackage{amsmath,amsthm}
\usepackage{amssymb,url}

\title{Weak Rudin-Keisler reductions on projective ideals}
\author{Konstantinos A. Beros}
\date{}
\address{Department of Mathematics, University of North Texas, General Academics Building 435, 1155 Union Circle, \#311430, Denton, TX 76203-5017}
\email{beros@unt.edu}
\thanks{Thanks to my colleague and friend Anush Tserunyan for the discussion that led to the formulation of Theorem~\ref{T6}.}

\usepackage{hyperref,color}

\hypersetup{
pdfview=fitH,
pdfpagemode=UseNone,
colorlinks=true,
allcolors=blue,
linktocpage=true
}

\theoremstyle{plain}

\newtheorem{theorem}{Theorem}[section]
\newtheorem{lemma}[theorem]{Lemma}
\newtheorem{corollary}[theorem]{Corollary}
\newtheorem{proposition}[theorem]{Proposition}
\newtheorem*{theorem*}{Theorem}
\newtheorem*{corollary*}{Corollary}
\newtheorem*{proposition*}{Proposition}

\def\thmref#1{\label{#1}{\theoremstyle{plain}\newtheorem*{refthm#1}{Theorem~\ref{#1}}}}

\theoremstyle{definition}

\newtheorem{definition}[theorem]{Definition}

\newtheorem{example}[theorem]{Example}

\theoremstyle{remark}

\newtheorem*{remark*}{Remark}
\numberwithin{equation}{section}

\def\upto{\upharpoonright}

\def\dom{{\rm dom}}

\def\leqB{\leq_{\rm B}}
\def\leqRK{\leq_{\rm RK}}
\def\leqwRK{\leq_{\rm wRK}}
\def\leqK{\leq_{\rm K}}

\def\I{\mathcal I}
\def\J{\mathcal J}
\def\U{\mathcal U}
\def\V{\mathcal V}
\def\sdif{\mathop{\vartriangle}}
\def\Fin{\mathsf{Fin}}
\def\Inf{\mathsf{Inf}}
\def\Coinf{\mathsf{Coinf}}
\def\Cof{\mathsf{Cof}}

\begin{document}



\begin{abstract}
We consider a slightly modified form of the standard Rudin-Keisler order on ideals and demonstrate the existence of complete (with respect to this order) ideals in various projective classes.  Using our methods, we obtain a simple proof of Hjorth's theorem on the existence of a complete $\mathbf \Pi^1_1$ equivalence relation.  Our proof of Hjorth's theorem enables us (under PD) to generalize his result to the classes of $\mathbf \Pi^1_{2n+1}$ equivalence relations.
\end{abstract}

\maketitle



\section{Introduction}\label{S1}

An {\em ideal} on $\omega$ is a family, $\I$, of subsets of $\omega$ such that, for any $x,y \subseteq \omega$, one has $x,y \in \I \implies x \cup y \in \I$ and $y \subseteq x \in \I \implies y \in \I$.  For the present purposes, we identify the power set of the natural numbers with the Cantor space $2^\omega$ and consider definable ideals on $\omega$ as subsets of $2^\omega$.  

One defines the Rudin-Keisler order, $\leqRK$, for ideals, by letting $\I \leqRK \J$ (for ideals $\I$ and $\J$ on $\omega$) iff there exists a function $f: \omega \rightarrow \omega$ such that 
\[
x \in \I \iff f^{-1} [x] \in \J,
\]
for each $x \in 2^\omega$.  (Here and throughout this paper, we use the notation $f[x]$ (respectively, $f^{-1} [x]$) to denote the $f$-image (respectively, the $f$-preimage) of the set $x \subseteq \omega$.)  The map $f$ is called a {\em Rudin-Keisler reduction} of $\I$ to $\J$.  The Rudin-Keisler order was first described by Mary Ellen Rudin and has been studied extensively since its introduction.  Subsequently, several variants of the Rudin-Keisler order have been considered.  Many of these are discussed in the survey \cite{HRUSAK.combinatorics.filters} by Michael Hru\u{s}\'{a}k or in Chapter 3 of Kanovei \cite{KAN.borel.equiv}.  In some cases, these variations of the Rudin-Keisler order coincide with the ordinary Rudin-Keisler order when restricted to the class of ultrafilters (or maximal ideals).  For instance, the Kat\v{e}tov order, $\leqK$, has this property, where $\I \leqK \J$ iff there is $f: \omega \rightarrow \omega$ such that $x \in \I \implies f^{-1} [x] \in \J$, for each $x \subseteq \omega$.  In what follows, we consider another modification of the Rudin-Keisler order which again agrees with the usual Rudin-Keisler order on the set of ultrafilters.  We will, however, be interested only in projective (hence non-maximal) ideals on $\omega$.

\subsection{The wRK-order}

The following is our principal definition, which we introduced in \cite{BEROS.universal}.

\begin{definition}
If $\I$ and $\J$ are ideals on $\omega$, we say that $\I$ is {\em weak Rudin-Keisler reducible} to $\J$ iff there is an infinite set $A \subseteq \omega$ and a function $f: A \rightarrow \omega$ such that $x \in \I \iff f^{-1} [x] \in J$, for each $x \in 2^\omega$.  We write $\I \leqwRK \J$ and call the map $f$ a {\em weak Rudin-Keisler reduction} of $\I$ to $\J$.

If $\mathcal C$ is a family of ideals and $\J \in \mathcal C$ is such that $\I \leqwRK \J$, for each $\I \in \mathcal C$, then we say that $\J$ is a {\em wRK-complete} $\mathcal C$ ideal on $\omega$.
\end{definition}

Note that, as with the usual Rudin-Keisler order, if $\U \leqwRK \V$ and $\V$ is an ultrafilter, then so is $\U$.  For the sake of completeness, we verify that $\leqRK$ and $\leqwRK$ coincide on the set of ultrafilters.

\begin{proposition}
If $\U$ and $\V$ are ultrafilters on $\omega$, then $\U \leqwRK \V$ iff $\U \leqRK \V$.
\end{proposition}

\begin{proof}
One half of the claim follows from the definitions.  For the other half, suppose that $\U \leqwRK \V$ via $f : A \rightarrow \omega$.  We wish to define a Rudin-Keisler reduction of $\U$ to $\V$.  Let $y = f[A]$ and note that $y \in \U$, since $f^{-1} [y] = A = f^{-1} [\omega] \in \V$.  Let $z \subseteq y$ be such that $z \in \U$ and $y \setminus z$ is infinite.  Such a $z$ exists by the maximality of $\U$.  Let $B = f^{-1}[z]$ and define $h: \omega \rightarrow \omega$ by letting $h \upto B = f\upto B$ and $h \upto (\omega \setminus B)$ be any fixed bijection between $\omega \setminus B$ and $y \setminus z$.  Since $\U$ and $\V$ are ultrafilters, to verify that $h$ is a Rudin-Keisler reduction of $\U$ to $\V$, it will suffice to check that $h^{-1} [x] \in \V$, whenever $x \in \U$.  Indeed, suppose that $x \in \U$.  We have
\[
h^{-1} [x] = h^{-1} [x\cap z] \cup h^{-1} [x \setminus z] \supseteq h^{-1} [x\cap z] = f^{-1} [ x \cap z ] \in \V,
\]
since $x,z \in \U$ and $f$ is a weak Rudin-Keisler reduction of $\U$ to $\V$.  It follows from the upward closure of $\V$ that $h^{-1} [x] \in \V$.
\end{proof}

\subsection{Factor algebras}

Observe that if $\I \leqwRK \J$, via $f$, there is an embedding of the Boolean algebra $2^\omega / \I$ into $2^\omega / \J$.  To see this, let $\pi_\I$ and $\pi_\J$ be the quotient maps onto $2^\omega / \I$ and $2^\omega / \J$, respectively.  One may define an injective homomorphism $\tilde f : 2^\omega / \I \rightarrow 2^\omega / \J$ by \[
\tilde f ( \pi_\I (x) ) = \pi_\J ( f^{-1} [x]).
\]
Thus, if $\J$ is a wRK-complete ideal for a class, $\mathcal C$, of ideals, then $2^\omega / \I$ embeds in $2^\omega / \J$, for each $\I \in \mathcal C$.  In other words, the algebra $2^\omega / \J$ is injectively universal for algebras of the form $2^\omega / \I$, with $\I \in \mathcal C$.

In \cite{KECHRIS.rigidity.properties}, Alexander Kechris discusses the ``Borel cardinality'' of a factor algebra of the form $2^\omega / \I$.  Kechris regards $2^\omega / \J$ as having greater Borel cardinality than $2^\omega / \I$ if there is an embedding $ \varphi : 2^\omega / \I \rightarrow 2^\omega / \J$ which has a Borel-measurable lifting to $2^\omega$.  That is, there exists a Borel-measurable map $\varphi^* : 2^\omega \rightarrow 2^\omega$ such that the diagram 
\[
\begin{array}{ccc}
2^\omega & \xrightarrow{\ \varphi^* \ } & 2^\omega\\
\hspace{1.4ex}\left\downarrow {\scriptstyle \pi_I} \rule[0ex]{0ex}{2.5ex} \right .& & \hspace{1.4ex}\left\downarrow {\scriptstyle \pi_J}\rule[0ex]{0ex}{2.5ex}\right .\\
2^\omega / \I & \xrightarrow{\;\ \varphi \ \;} & 2^\omega / \J
\end{array}
\]
commutes.  Equivalently, $2^\omega / \J$ has greater Borel cardinality than $2^\omega / \I$ iff there is a Borel-measurable map $h : 2^\omega \rightarrow 2^\omega$ such that $x \sdif y \in \I \iff h(x) \sdif h(y) \in \J$, for all $x,y \in 2^\omega$, where ``$x\sdif y$'' denotes the symmetric difference of $x$ and $y$. 

Since all maps of the form $x \mapsto f^{-1}[x]$ (for $f : \omega \rightarrow \omega$) are continuous on $2^\omega$ (in particular, Borel-measurable), our notion of $\I \leqwRK \J$ implies that $2^\omega / \J$ has greater Borel cardinality than $2^\omega / \I$.  It follows that if $\J$ is a wRK-complete ideal for some class, $\mathcal C$, of ideals, then $2^\omega / \J$ has maximum Borel cardinality among factor algebras by ideals in $\mathcal C$.

\subsection{wRK-complete ideals}

In Section~\ref{S3}, we prove our principle results.  These establish the existence of wRK-complete ideals for the classes of $\mathbf \Sigma^1_n$ ideals and $\mathbf \Pi^1_1$ ideals.

\begin{theorem}\thmref{T1}
For each $n \in \omega$, there is a wRK-complete $\mathbf \Sigma^1_n$ ideal on $\omega$.
\end{theorem}

\begin{theorem}\thmref{T2}
There is a wRK-complete $\mathbf \Pi^1_1$ ideal on $\omega$.
\end{theorem}

The proof of Theorem~\ref{T2} uses the prewellordering property of the pointclass $\mathbf \Pi^1_1$ as well as closure under universal quantification over the reals and closure under countable unions and intersections.  As such, the proof carries over to any other ranked pointclass with similar closure properties.  In particular, assuming projective determinacy (PD), we obtain the following corollary to the proof of Theorem~\ref{T2}.

\begin{corollary}[PD]
For each $n$, there is a wRK-complete $\mathbf \Pi^1_{2n+1}$ ideal on $\omega$.
\end{corollary}

In Section~\ref{S4}, we prove similar results for ideals which are ``nontrivial'' in a certain sense.

\begin{definition}
An ideal $\I \subseteq 2^\omega$ is {\em proper} iff $\I \neq 2^\omega$ and $\I$ contains all finite subsets of $\omega$.
\end{definition}

We have the following result.

\begin{theorem}\thmref{T7}
There is a wRK-complete proper uncountable $\mathbf \Pi^1_1$ ideal on $\omega$.
\end{theorem}

\subsection{Projective equivalence relations}

There are anologies between the theory of equivalence relations and that of ideals.  In many cases, arguments in the context of ideals carry over to equivalence relations.  The proof of Theorem~\ref{T2} is an example of this.  

We begin by establishing some terminology.

Suppose that $E$ and $F$ are equivalence relations on a Polish space $X$.  A map $f : X \rightarrow X$ is a {\em reduction} of $E$ to $F$ iff, for each pair $x,y \in X$, one has $x E y \iff f(x) F f(y)$.  In general, one restricts attention to those reductions which are at least Borel-measurable.  In the case that there is a Borel-measurable reduction of $E$ to $F$, we say that $E$ is {\em Borel-reducible} to $F$ and write $E \leqB F$.  If $\mathcal C$ is a class of equivalence relations and $F \in \mathcal C$ is such that $E \leqB F$, for each $E \in \mathcal C$, then we say that $F$ is a {\em complete} $\mathcal C$ equivalence relation.

Hjorth \cite{HJORTH.universal.co-analytic} proved that there is a $\Pi^1_1$ (effectively co-analytic) equivalence relation on $2^\omega$ which is a complete co-analytic equivalence relation.  The key step in Hjorth's proof was the following parameterization theorem for co-analytic equivalence relations.

\begin{theorem}[Hjorth, 1996]\thmref{T5}
There is a universal $\Pi^1_1$ set for $\mathbf \Pi^1_1$ equivalence relations, i.e., there is a (lightface) $\Pi^1_1$ set $\mathcal E \subseteq 2^\omega \times 2^\omega \times 2^\omega$ such that the cross-sections $\mathcal E_\tau$ are exactly the $\mathbf \Pi^1_1$ equivalence relations on $2^\omega$.
\end{theorem}

Hjorth's proof involves admissable ordinals and critical use of the Lusin-Sierpinski Theorem (see Theorem 4A.4 in Moschovakis \cite{MOSCHOVAKIS.dst}) and is thus intimately connected with the effective theory of $\Pi^1_1$ sets.  Even under PD, this theory is not known to generalize to the classes $\Pi^1_{2n+1}$.  By combining elements of the proof of our Theorem~\ref{T2} with an $s$-$n$-$m$ Theorem argument, we obtain Hjorth's theorem without use of any effective theory specific to the class $\Pi^1_1$.  In fact, our argument generalizes, assuming PD, to yield a universal $\Pi^1_{2n+1}$ set for $\mathbf \Pi^1_{2n+1}$ equivalence relations.  In Section~\ref{S5}, we prove the following result.

\begin{theorem}[PD]\thmref{T4}
There is a universal $\Pi^1_{2n+1}$ set for $\mathbf \Pi^1_{2n+1}$ equivalence relations on $2^\omega$.
\end{theorem}

As in Hjorth \cite{HJORTH.universal.co-analytic}, the existence of complete $\mathbf \Pi^1_{2n+1}$ equivalence relation follows immediately from Theorem~\ref{T4}.  Indeed, if $\mathcal E$ is a universal set for $\mathbf \Pi^1_{2n+1}$ equivalence relations on $2^\omega$, one may define a complete $\mathbf \Pi^1_{2n+1}$ equivalence relation, $F$, on $2^\omega \times 2^\omega \approx 2^\omega$ by 
\[
(\sigma , x) F (\tau , y) \iff \tau = \sigma \ \& \ (x,y) \in \mathcal E_\tau.
\]
If $E$ is any $\mathbf \Pi^1_{2n+1}$ equivalence relation on $2^\omega$, with $E = \mathcal E_\tau$, the (continuous) map $x \mapsto (\tau , x)$ reduces $E$ to $F$.



\section{Preliminaries and notation}\label{S2}

Our principle references are Kechris~\cite{KECHRIS.dst} and Moschovakis~\cite{MOSCHOVAKIS.dst}.  Our notation is largely the same as theirs.  We review some key facts and terminology below.

For sets $X$, $Y$, and $A \subseteq X \times Y$, if $x \in X$, let $A_x$ denote the vertical cross-section, $\{ y \in Y : (x,y) \in A\}$, of $A$.  As mentioned above, if $f : X \rightarrow Y$ is any function and $x \subseteq X$, we let $f[x] = \{ f(a) : a \in x\}$ and, if $y \subseteq Y$, we let $f^{-1} [y] = \{ a : f(a) \in y\}$.  We will freely identify $\mathcal P( \omega)$ with the Cantor space $2^\omega$.  For example, we regard $\emptyset$ and $\omega$ as the elements $\bar 0$ and $\bar 1$ in the Cantor space.

\subsection{Classical notions}

If $\Gamma$ is a pointclass, we say that $\Gamma$ is {\em $2^\omega$-parameterized} iff, for each Polish space $X$, there is a $\Gamma$-set $\mathcal U \subseteq 2^\omega \times X$ such that, for each $A \subseteq X$ with $A \in \Gamma$, there exists $\tau \in 2^\omega$ with $A = \mathcal U_\tau$.  Such a set $\mathcal U$ is called a {\em universal $\Gamma$-set for $X$}.  Each of the projective classes $\mathbf \Sigma^1_n$ and $\mathbf \Pi^1_n$ is $2^\omega$-parameterized (see Theorem~37.7 in Kechris~\cite{KECHRIS.dst}).

Suppose $\Gamma$ is a pointclass, $\Gamma$ is {\em ranked} iff, for each Polish space $X$ and each $A \subseteq X$, with $A \in \Gamma$, there is a {\em $\Gamma$-rank} $\varphi : A \rightarrow {\rm ORD}$.  That is, there are relations $S , P \subseteq X^2$ in $\Gamma$ and $\bar{\Gamma}$, respectively, such that, for each $y \in A$, 
\[
x \in A \ \& \ \varphi (x) \leq \varphi (y) \iff S(x,y) \iff P(x,y),
\]
for each $x \in X$.  (Here $\bar{\Gamma}$ denotes the class of complements of $\Gamma$-sets.)  In other words, the predicate ``$x \in A \ \& \ \varphi (x) \leq \varphi (y)$'' is uniformly in $\Gamma \cap \bar{\Gamma}$, provided that $y \in A$.  See \S34 in Kechris~\cite{KECHRIS.dst} for an in-depth treatment of $\Gamma$-ranks.  It is a fundamental result that $\mathbf \Pi^1_1$ is a ranked pointclass.  Assuming PD, the classes $\mathbf \Pi^1_{2n+1}$ are also ranked.  See Theorems~34.4 and 39.2 in Kechris~\cite{KECHRIS.dst} for proofs of these facts.  The notation ``$x \leq_{\varphi} y$'' is a shorthand to indicate ``$\varphi (x) \leq \varphi (y)$'', when $\varphi$ is a $\Gamma$-rank on a set $A$ and $x,y \in A$.

\subsection{Effective theory}

For each classical, or ``boldface'', pointclass $\mathbf \Gamma$ (e.g., $\mathbf \Sigma^1_1$, $\mathbf \Pi^1_1$, etc.)~on a recursively presented Polish space, we let $\Gamma$ denote its effective, or ``lightface'', counterpart, (e.g., $\Sigma^1_1$, $\Pi^1_1$, etc.).  The effective pointclasses share many of the properties of their classical brethren.  For instance, the pointclass $\Pi^1_1$ is also ranked.  (See Theorem~4B.2 in Moschovakis \cite{MOSCHOVAKIS.dst}.)  For a $\Pi^1_1$-rank, $\varphi$, the predicates $S$ and $P$ as above will, in fact, be in the (lightface) classes $\Sigma^1_1$ and $\Pi^1_1$, respectively.  As in the classical case, PD guarantees the lightface classes $\Pi^1_{2n+1}$ are ranked as well.  (See Theorem~6B.1 in Moschovakis~\cite{MOSCHOVAKIS.dst}.)

The other tools we require from the effective theory are the existence of {\em good universal systems} and the resulting $s$-$n$-$m$ Theorem.  We say that a Polish space, $X$, is a {\em product space} if $X$ is a product of at most countably many copies of $2^\omega$.  (This is a restricted definition compared to that in Moschovakis~\cite{MOSCHOVAKIS.dst}.)

\begin{definition}
A family of sets $\{ \mathcal U^X \subseteq 2^\omega \times X : X \mbox{ is a product space}\}$ is a {\em good universal system} for a (boldface) pointclass $\mathbf \Gamma$ iff the following hold:
\begin{enumerate}
\item Each $\mathcal U^X$ is a (lightface) $\Gamma$-set;
\item For each product space $X$, the set $\mathcal U^X$ is universal for $\mathbf \Gamma$-subsets of $X$, i.e., the cross-sections $\mathcal U^X_\tau$ are exactly the $\mathbf \Gamma$-subsets of $X$;
\item If $A \subseteq X$ is a (lightface) $\Gamma$-set, there is a recursive $\varepsilon \in 2^\omega$ such that $A = \mathcal U^X_\varepsilon$;
\item\label{snm} For each pair, $X_1 , X_2$, of product spaces, there is a recursive function $s: 2^\omega \times X_1 \rightarrow 2^\omega$ such that, for each $\alpha \in 2^\omega$ and $x \in X_1$, 
\[
(\alpha , x , y) \in \mathcal U^{X_1 \times X_2} \iff (s(\alpha , x) , y) \in \mathcal U^{X_2}.
\] 
\end{enumerate}
\end{definition}

Property~\ref{snm} is also known as the $s$-$n$-$m$ Theorem.  

It follows from Theorem~3H.1 in Moschovakis~\cite{MOSCHOVAKIS.dst}, that each of the projective classes $\mathbf \Sigma^1_n$ and $\mathbf \Pi^1_n$ has a good universal system.  For illustrative purposes, we now describe a typical application of good universal systems.

\begin{example}
Let $\mathbf \Gamma$ be a pointclass, closed under finite intersections, with an associated good universal system $\mathcal U^X$.  There is a recursive function $f : (2^\omega)^2 \rightarrow 2^\omega$ such that, for $\sigma , \tau \in 2^\omega$, we have $\mathcal U^{2^\omega}_{f(\sigma , \tau)} = \mathcal U^{2^\omega}_\sigma \cap \mathcal U^{2^\omega}_\tau$.

To find such an $f$, let $s : (2^\omega)^3 \rightarrow 2^\omega$ be the recursive function from property~\ref{snm} above, for the product spaces $X_1 = 2^\omega \times 2^\omega$ and $X_2 = 2^\omega$, and let $\varepsilon \in 2^\omega$ be recursive such that, for $\sigma, \tau , x \in 2^\omega$, 
\begin{align*}
(\sigma , x) \in \mathcal U^{2^\omega} \ \& \ (\tau  , x) \in \mathcal U^{2^\omega} 
&\iff (\varepsilon , \sigma , \tau , x) = \mathcal U^{2^\omega \times 2^\omega \times 2^\omega}\\
&\iff (s(\varepsilon , \sigma , \tau) , x) \in \mathcal U^{2^\omega}
\end{align*}
Define $f(\sigma , \tau) = s (\varepsilon, \sigma, \tau)$.  This is the desired recursive function.
\end{example}



\section{Projective ideals on $\omega$}\label{S3}

\subsection{$\mathbf \Sigma^1_n$ ideals}

We now restate and prove Theorem~\ref{T1} above.

\begin{refthmT1}
For each $n \in \omega$, there is a wRK-complete $\mathbf \Sigma^1_n$ ideal on $\omega$.
\end{refthmT1}

\begin{proof}
Fix $n \in \omega$.  The proof is the same for each $n$.  First of all, let $\{ A_p : p \in 2^{<\omega}\}$ be a partition of $\omega$ into infinite sets and let $h_p : \omega \rightarrow A_p$ be fixed bijections.  Let $\mathcal U  \subseteq 2^\omega \times 2^\omega$ be a universal $\mathbf \Sigma^1_n$ set for $2^\omega$.  Define a $\mathbf \Sigma^1_n$ set $F \subseteq 2^\omega$ by letting $x \in F$ iff there exists $\tau \in 2^\omega$ such that 
\begin{enumerate}
\item\label{it1:1} $(\forall p \in 2^{<\omega}) (p \not \subset \tau \implies x \cap A_p = \emptyset)$,
\item\label{it1:2} $h_\emptyset^{-1} [x] \in \mathcal U_\tau$ and
\item\label{it1:3} $(\forall p \in 2^{<\omega})(p \subset \tau \implies h_p^{-1} [x] = h_\emptyset^{-1} [x])$.
\end{enumerate}
Let $\J$ be the ideal generated by $F$.  Since the class $\mathbf \Sigma^1_n$ is closed under continuous images and countable unions, it follows that $\J$ is also $\mathbf \Sigma^1_n$. 

To see that $\J$ is a wRK-complete $\mathbf \Sigma^1_n$ ideal, suppose that $\I$ is an arbitrary $\mathbf \Sigma^1_n$ ideal and $\tau \in 2^\omega$ is such that $\I = \mathcal U_\tau$.  Define a map $f : \bigcup_{p \subset \tau} A_p \rightarrow \omega$ by $f \upto A_p = h_p^{-1}$.  This map is well-defined because the $A_p$ are pairwise disjoint.

Suppose first that $x \in \I$.  To see that $y = f^{-1} [x] \in \J$, it will suffice to verify conditions \ref{it1:1}-\ref{it1:3} above and conclude that $y \in F$.  That $y$ satisfies condition \ref{it1:1} follows from the observation that, for each $p \not\subset \tau$, since $\dom (f) \cap A_p = \emptyset$, we have $y \cap A_p = \emptyset$.  Also, for each $p \subset \tau$, we have $y \cap A_p = h_p [x]$ and therefore
\[
h_\emptyset^{-1} [y] = h_\emptyset^{-1} [y \cap A_\emptyset] = h_\emptyset^{-1} [h_\emptyset[x]] = x \in \I = \mathcal U_\tau.
\]
This shows that $y$ satisfies condition \ref{it1:2}.  Finally, the definition of $f$ guarantees that $y$ satisfies condition \ref{it1:3}.

Suppose now that $y = f^{-1} [x] \in \J$.  We wish to see that $x \in \I$.  Let $y_0 , \ldots , y_k \in F$ be such that $y \subseteq y_0 \cup \ldots \cup y_k$.  Let $\tau_0 , \ldots , \tau_k \in 2^\omega$ be as in the definition of $F$, such that each $\tau_i$ witnesses the membership of $y_i$ in $F$.  Let $m \in \omega$ be large enough that $\tau_i \upto m \neq \tau \upto m$, for each $\tau_i \neq \tau$.  Now define $x_0 , \ldots , x_k$ by letting $x_i = h_{\tau \upto m}^{-1} [y_i]$, for each $i \leq k$.  Note that $x_i = \emptyset$, for $i$ with $\tau_i \neq \tau$, and that $x_i \in \mathcal U_\tau = \I$, for each $i \leq k$.  Finally, observe that 
\[
x = h_{\tau \upto m}^{-1} [y] \subseteq \bigcup_{i \leq k} h_{\tau\upto m}^{-1} [y_i] = x_0 \cup \ldots \cup x_k.
\]
As $\I$ is an ideal, we conclude that $x \in \I$.

We have shown that $f$ is a weak Rudin-Keisler reduction of $\I$ to $\J$.  This concludes the proof.
\end{proof}

\subsection{Co-analytic ideals}

We now prove our main result for co-analytic ideals on $\omega$.

\begin{refthmT2}
There is a wRK-complete $\mathbf \Pi^1_1$ ideal on $\omega$.
\end{refthmT2}

The key lemma in the proof of Theorem~\ref{T2} is a parameterization of the $\mathbf \Pi^1_1$ ideals on $\omega$.

\begin{lemma}\label{L1}
There is a universal set for $\mathbf \Pi^1_1$ ideals on $\omega$, i.e., there is a $\mathbf \Pi^1_1$ set $\mathcal A \subseteq 2^\omega \times 2^\omega$ such that the cross-sections $\mathcal A_\tau$ are exactly the $\mathbf \Pi^1_1$ ideals on $\omega$.
\end{lemma}

\begin{proof}
We proceed inductively to define $\mathbf \Pi^1_1$ sets $\mathcal A^{(0)} \supseteq \mathcal A^{(1)} \supseteq \ldots$ such that $\mathcal A = \bigcap_n \mathcal A^{(n)}$ is a universal set for $\mathbf \Pi^1_1$ ideals on $\omega$.  Along the way, we will also select $\mathbf \Pi^1_1$-ranks $\varphi_n : \mathcal A^{(n)} \rightarrow \omega_1$.  

Let $\mathcal U \subseteq 2^\omega \times 2^\omega$ be a universal $\mathbf \Pi^1_1$ set.  As the base case of our induction, let $\mathcal A^{(0)} = \mathcal U \cup (2^\omega \times \{ \emptyset \})$.  Given the $\mathbf \Pi^1_1$ set $\mathcal A^{(n)}$, we describe how to define $\mathcal A^{(n+1)}$.  Let $\varphi_n : \mathcal A^{(n)} \rightarrow \omega_1$ be a $\mathbf \Pi^1_1$-rank and take $\mathcal A^{(n+1)} \subseteq 2^\omega \times 2^\omega$ to be the set of all $(\tau , x)$ such that either $x = \emptyset$ or the following hold:
\begin{enumerate}
\item\label{it2:1} $x \in \mathcal A^{(n)}_\tau$,
\item\label{it2:2} $(\forall y) (y \subseteq x \implies y \in \mathcal A^{(n)}_\tau)$ and
\item\label{it2:3} $(\forall y) ((y \in \mathcal A^{(n)}_\tau \ \& \ (\tau , y) \leq_{\varphi_n} (\tau , x)) \implies x \cup y \in \mathcal A^{(n)})$.
\end{enumerate}

In essence, we are defining $\mathcal A^{(n+1)}_\tau$ by removing from $\mathcal A^{(n)}_\tau$ those $x$ for which there exist lower ranked witnesses to the failure of closure under finite unions.  It follows from the definability properties of $\mathbf \Pi^1_1$-ranks that $\mathcal A^{(n+1)}$ is $\mathbf \Pi^1_1$.

Let $\mathcal A = \bigcap_n \mathcal A^{(n)}$.  Since each $\mathcal A^{(n)}$ is $\mathbf \Pi^1_1$, so is $\mathcal A$.  First observe that, if $\mathcal U_\tau$ is already an ideal, then $\mathcal A^{(n)}_\tau = \mathcal U_\tau$, for each $n$, since the process of producing $\mathcal A^{(n+1)}$ from $\mathcal A^{(n)}$ removes no elements from $\mathcal A^{(n)}_\tau$ in this case.  It follows that each $\mathbf \Pi^1_1$ ideal on $\omega$ appears as some cross-section $\mathcal A_\tau$.  

It now only remains to see that every cross-section $\mathcal A_\tau$ is an ideal on $\omega$.  Indeed, suppose that $x,y \in \mathcal A_\tau$.  Fix $n \in \omega$.  We have that $x,y \in \mathcal A^{(n+1)}_\tau$ and we may assume that $(\tau , x) \leq_{\varphi_n} (\tau , y)$.  Condition \ref{it2:3} in the definition of $\mathcal A^{(n+1)}$ thus implies that $x \cup y \in \mathcal A^{(n)}_\tau$.  Similarly, the membership of $x$ in $\mathcal A^{(n+1)}$ implies that every subset of $x$ is a member of $\mathcal A^{(n)}_\tau$, by condition \ref{it2:2}.  As $n$ was arbitrary, it follows that $x \cup y \in \mathcal A_\tau$ and that every subset of $x$ is a member of $\mathcal A_\tau$.  In short, $\mathcal A_\tau$ is an ideal.
\end{proof}

Before proceeding with the proof of Theorem~\ref{T2}, we remark that an argument analogous to the one used to produce a complete $\mathbf \Pi^1_1$ equivalence relation from a universal set for $\mathbf \Pi^1_1$ equivalence relations does not work for ideals, as the object one would obtain is not itself be an ideal.  We thus require a means of ``coding'' all $\mathbf \Pi^1_1$ ideals into a single $\mathbf \Pi^1_1$ ideal.

\begin{proof}[Proof of Theorem~\ref{T2}]
Let $\{ A_p : p \in 2^{<\omega} \}$ be a partition of $\omega$ into infinite sets and let $h_p : \omega \rightarrow A_p$ be fixed bijections.  Take $\mathcal A \subseteq 2^\omega \times 2^\omega$ to be a universal set for $\mathbf \Pi^1_1$ ideals, the existence of which is guaranteed by Lemma~\ref{L1}.  Define
\[
\J = \{ x \in 2^\omega : (\forall \tau \in 2^\omega) (\forall^\infty p \subset \tau) (h_p^{-1} [x] \in \mathcal A_\tau) \}.
\]

In the first place, we wish to see that $\J$ is an ideal.  That $\J$ is downward closed follows since each $\mathcal A_\tau$ is downward closed and preimages preserve containment.  To verify closure under unions, fix $x, y \in \J$ and $\tau \in 2^\omega$, and let $p_0 \subset \tau$ be such that $h_p^{-1} [x] , h_p^{-1} [y] \in \mathcal A_\tau$, for each $p \subset \tau$, with $p \supseteq p_0$.  For each $p \subset \tau$, if $p \supseteq p_0$, we have
\[
h_p^{-1} [x \cup y] = h_p^{-1} [x] \cup h_p^{-1} [y] \in \mathcal A_\tau,
\]
since $\mathcal A_\tau$ is an ideal.  In other words, $h_p^{-1} [x \cup y] \in \mathcal A_\tau$, for all but finitely many $p \subset \tau$.  Since $\tau$ was arbitrary, we conclude that $x \cup y \in \J$.

To see that $\J$ is a wRK-complete $\mathbf \Pi^1_1$ ideal, suppose that $\I = \mathcal A_\tau$ is a fixed $\mathbf \Pi^1_1$ ideal on $\omega$.  Define a map $f : \bigcup_{p \subset \tau} A_p \rightarrow \omega$ by $f \upto A_p = h_p^{-1}$.  As before, $f$ is well-defined by the disjointness of the $A_p$.  Also, note that $f^{-1} [x] \cap A_p = h_p [x]$, for each $p \subset \tau$.  Thus,
\[
h_p^{-1} [f^{-1} [x]] = h_p^{-1} [f^{-1} [x] \cap A_p] = h_p^{-1} [h_p [x]]  = x
\]
and hence $x \in \I$ iff $h_p^{-1} [f^{-1}[x]] \in \mathcal A_\tau$, for each $p \subset \tau$, since $\I = \mathcal A_\tau$.  This, combined with the observation that $f^{-1} [x] \cap A_p = \emptyset$, for each $p \not \subset \tau$, shows that $x \in \I$ iff $f^{-1} [x] \in \J$.
\end{proof}

\subsection{Free abelian subgroups}

As an aside, we show how the method of the proof of Theorem~\ref{T1} may be used to produce free abelian subgroups in a large class of Polish groups.  Moreover, the free subgroups we describe have $\mathfrak c$-many definable homomorphisms into $\mathbb Z$.  Specifically, we have the following result.

\begin{theorem}\label{T6}
Let $G$ be a Polish group with an element of infinite order.  The countable power $G^\omega$ contains a $K_\sigma$ subgroup, $H$, which is free abelian on $\mathfrak c$-many generators and such that there exist $\mathfrak c$-many Borel-measurable homomorphisms $\varphi : H \rightarrow \mathbb Z$.
\end{theorem}

It is a theorem of N\"obeling \cite{NOBELING} that the group, $\mathcal B$, of bounded sequences in $\mathbb Z^\omega$ is free.  This result, however, requires an essesntial use of the axiom of choice.  Indeed, Andreas Blass \cite{BLASS.nobelings.group} has shown (under ZFC) that there are only countably many Borel-measurable homomorphisms $\varphi : \mathcal B \rightarrow \mathbb Z$.  (In this case, Borel-measurability is equivalent to the statement that $\varphi^{-1}[\{n\}]$ is Borel for each $n \in \omega$.)  Blass further showed that (under AD) there are only countably many homomorphisms $\varphi : \mathcal B \rightarrow \mathbb Z$ (Borel-measurable or otherwise).  On the other hand, for any free abelian group $G$, there must be $\mathfrak c$ many homomorphisms $\varphi : \mathcal B \rightarrow \mathbb Z$.  Thus, Blass' result shows that, under AD, $\mathcal B$ is not free.  

The subgroup we construct below, however, has a free basis of a sufficiently explicit nature that there are $\mathfrak c$-many definable homomorphisms into $\mathbb Z$.

\begin{proof}[Proof of Theorem~\ref{T6}]
Fix a Polish group $G$, with identity element $e$ and $g \in G$ such that the powers of $x$ are all distinct.  Let $p \mapsto n_p$ be a bijection of $2^{<\omega}$ onto $\omega$.  Define a set $S \subseteq G^\omega$ by letting $x \in S$ iff there exists $\tau \in 2^\omega$ such that
\begin{enumerate}
\item\label{it3:1} $(\forall p \in 2^{<\omega}) (p \not \subset \tau \implies x (n_p) = e)$ and
\item\label{it3:2} $(\forall p \in 2^{<\omega}) (p \subset \tau \implies x(n_p) = g)$.
\end{enumerate}
For $\tau \in 2^\omega$, let $x_\tau$ denote the unique element of $S$ whose membership in $S$ is witnessed by $\tau$.  Take $H$ to be the subgroup generated by $S$.  

In the first place, observe that $S$ is the projection onto the second coordinate of the compact set
\[
\{ (\tau , x) \in 2^\omega \times \{e,g\}^\omega : (\forall p \subset \tau) (x(n_p) = g) \wedge (\forall p \not\subset \tau)(x(n_p) = e)\}
\] 
and is itself compact. Hence, $H$ is compactly generated and therefore $K_\sigma$. 

To see that $H$ is a free abelian group, it suffices to show that there are no relations among the elements of $S$, besides those dictated by commutativity.  Indeed, suppose that $x_{\sigma_0} , \ldots , x_{\sigma_m} , x_{\tau_0} , \ldots , x_{\tau_n} \in S$, with $\{ \sigma_0 , \ldots , \sigma_m\}$ and $\{\tau_0 , \ldots , \tau_n \}$ sets of pairwise distinct reals, and $i_0 , \ldots , i_m , j_0 , \ldots , j_n \in \mathbb Z$ (all nonzero) are such that 
\[
y = x_{\sigma_0}^{i_0} \cdot \ldots \cdot x_{\sigma_m}^{i_m} = x_{\tau_0}^{j_0} \cdot \ldots \cdot x_{\tau_n}^{j_m} = z.
\]

Our first claim is that $\{ \sigma_0 , \ldots , \sigma_m\} = \{\tau_0 , \ldots , \tau_n\}$.  Suppose that this was not the case.  For instance, if $\sigma_0 \notin \{ \tau_0 , \ldots , \tau_n\}$, then we may choose $k \in \omega$ large enough that 
\[
\sigma_0 \upto k , \tau_0 \upto k , \ldots , \tau_n \upto k
\]
are all distinct and
\[
\sigma_0 \upto k , \ldots , \sigma_m \upto k
\]
are distinct as well.  Then $y(\sigma_0 \upto k) = g^{i_0}\neq e$, but $z( \sigma_0 \upto k) = e$, since $\sigma_0 \upto k \not\subset \tau_p$, for each $p \leq n$.  This is a contradiction.

We may also assume that $\sigma_p = \tau_p$, for each $p \leq n$.  Observe now that, if $k$ is as above, we have, for each $p \leq n$, 
\[
g^{i_p} = y(\sigma_p \upto k) = z(\sigma_p \upto k) = g^{j_p}.
\]
Since $g$ has infinite order, we conclude that $i_p = j_p$, for each $p \leq n$.  This shows that $S$ is a free basis for $H$.

Finally, we exhibit $\mathfrak c$-many Borel-measurable homomorphisms $\varphi : H \rightarrow \mathbb Z$.  Since $S$ is a free basis for $H$, it will suffice to define these homomorphisms on $S$ and then extend to $H$.  For each $\alpha \in 2^\omega$, let $\varphi_\alpha \upto S$ be given by
\[
\varphi_\alpha (x_\tau) =\begin{cases}
1 &\mbox{if } \tau = \alpha, \\
0 &\mbox{otherwise}.
\end{cases}
\]
We see that $\{ \varphi_\alpha : \alpha \in 2^\omega\}$ is a family of $\mathfrak c$-many distinct group homomorphisms of $H$ into $\mathbb Z$.  To verify that each $\varphi_\alpha$ is Borel-measurable, observe that, for each $n \in \mathbb Z$,
\begin{align*}
\varphi_\alpha^{-1} [\{ n\}] = \{ x \in H : (\exists k \in \omega)& (\exists s \in \mathbb Z^{k+1}) (\exists \tau_0 , \ldots , \tau_k \in 2^\omega) (\forall i \leq k) (\tau_i \neq \alpha\\
&\& \ x = x_\alpha^n \cdot x_{\tau_0}^{s(0)} \cdot \ldots \cdot x_{\tau_k}^{s(k)}) \}.
\end{align*}
In particular, $\varphi^{-1}[\{ n \}]$ is analytic, for each $n \in \omega$.  On the other hand, every $\mathbf \Sigma^1_1$-measurable map between Borel spaces is Borel-measurable and hence each $\varphi_\alpha$ is Borel-measurable.  We cannot, however, use Pettis' Theorem to conclude that the $\varphi_\alpha$ are continuous, since $H$ is not a Baire group.
\end{proof}



\section{Parameterizing co-analytic proper ideals}\label{S4}

We turn our attention to {\em proper} ideals, that is, ideals $\I$ such that $\I \neq 2^\omega$ and $\I$ contains all finite subsets of $\omega$ (equivalently, $\bigcup \I = \omega$).  Before proceeding, we recall some standard notation.
\begin{enumerate}
\item $\Fin = \{ x \in 2^\omega : (\forall^\infty n) (x(n) = 0)\}$
\item $\Inf = 2^\omega \setminus \Fin$
\item $\Cof = \{ x \in 2^\omega : \omega \setminus x \in \Fin\}$
\item $\Coinf = 2^\omega \setminus \Cof$
\end{enumerate}

Theorem~\ref{T3} is an elaboration of the method of Lemma~\ref{L1} which produces a universal set for proper uncountable co-analytic ideals.

\begin{theorem}\label{T3}
There is a universal set for proper uncountable co-analytic ideals.
\end{theorem}

\begin{proof}
For $\tau \in 2^\omega$, let $(\tau)_0 = \{ n : 2n \in \tau\}$ and $(\tau)_1 = \{ n : (2n+1) \in \tau\}$, and let $f : 2^\omega \rightarrow \Inf\setminus\Cof$ be a Borel surjection.  Let $\mathcal U \subseteq 2^\omega \times 2^\omega$ be a universal co-analytic set.  We define co-analytic sets $\mathcal V^{(0)} \supseteq \mathcal V^{(1)} \supseteq \ldots $, by induction, such that each $\mathcal V^{(n)}_\tau$ contains $\mathcal P (f((\tau)_0))$.  Let $\mathcal V^{(0)} \subseteq 2^\omega \times 2^\omega$ be the set of $(\tau,x) \in 2^\omega \times 2^\omega$ such that
\begin{enumerate}
\item $x \subseteq f((\tau)_0)$ or
\item $x \in \mathcal U_{(\tau)_1}$ and $x$ is co-infinite.
\end{enumerate}
Given $\mathcal V^{(n)} \subseteq 2^\omega \times 2^\omega$, let $\varphi_n : \mathcal V^{(n)} \rightarrow \omega_1$ be a $\mathbf \Pi^1_1$-rank.  Define $\mathcal V^{(n+1)}\subseteq 2^\omega \times 2^\omega$ by $(\tau,x) \in \mathcal V^{(n+1)}$ iff either $x \subseteq f((\tau)_0)$ or
\begin{enumerate}
\item\label{it4:1} $x \in \mathcal V^{(n)}_\tau$,
\item\label{it4:2} $(\forall y) (y \subseteq x \implies z \in \mathcal V^{(n)}_\tau)$,
\item\label{it4:3} $(\forall y) (y \subseteq f((\tau)_0) \implies x \cup y \in \mathcal V^{(n)}_\tau)$ and
\item\label{it4:4} $(\forall y) ((y \in \mathcal V^{(n)}_\tau \ \& \ (\tau,y) \leq_{\varphi_n} (\tau,x)) \implies x \cup y \in \mathcal V^{(n)}_\tau)$.
\end{enumerate}
By the properties of $\mathbf \Pi^1_1$-ranks, each $\mathcal V^{(n)}$ is $\mathbf \Pi^1_1$.  Let 
\[
\mathcal V = \{ (\tau,x) : (\exists u \in \Fin) (\forall n) (u \sdif x \in \bigcap_n \mathcal V^{(n)}_\tau)\}
\]
and note that $\mathcal V$ is $\mathbf \Pi^1_1$.

What follows will establish that the cross-sections $\mathcal V_\tau$ are exactly the uncountable co-analytic proper ideals on $\omega$.

First, we verify that each $\mathcal V_\tau$ is an ideal.  Since $\mathcal V_\tau$ is the set of finite variants of members of $\bigcap_n \mathcal V^{(n)}_\tau$, it will suffice to show that $\bigcap_n \mathcal V^{(n)}_\tau$ is an ideal.  

To verify downward closure, fix $y \subseteq x \in \bigcap_n \mathcal V^{(n)}_\tau$.  In the case that $x \subseteq f((\tau)_0)$, we have $y \in \mathcal (P(f((\tau)_0)) \subseteq \mathcal V^{(n)}_\tau$, for each $n$.  Hence $y \in \bigcap_n \mathcal V^{(n)}_\tau$.  In the case that $x \nsubseteq f((\tau)_0)$, part \ref{it4:2} of the definition of $\mathcal V^{(n+1)}$ from $\mathcal V^{(n)}$ implies that $y \in \mathcal V^{(n)}_\tau$, for each $n$.

Suppose that $x,y \in \bigcap_n \mathcal V^{(n)}_\tau$.  We wish to see that $x \cup y \in \bigcap_n \mathcal V^{(n)}_\tau$.  If $x,y \subseteq f((\tau)_0)$, then $x \cup y \in \mathcal P (f((\tau)_0)) \subseteq \bigcap_n \mathcal V^{(n)}_\tau$.  If $x \subseteq f((\tau)_0)$ and $x \nsubseteq f((\tau)_0)$, then part \ref{it4:3} of the definition of $\mathcal V^{(n+1)}$ from $\mathcal V^{(n)}$ implies that $x \cup y \in \mathcal V^{(n)}_\tau$.  We argue similarly if $x \nsubseteq f((\tau)_0)$ and $y \subseteq f((\tau)_0)$.  Finally, if $x,y \nsubseteq f((\tau)_0)$, then, for each $n$, we have $x,y \in V^{(n+1)}_\tau$ and, with no loss of generality, assume that $y \leq_{\varphi_n} x$.  Part \ref{it4:4} of the definition above implies that $x \cup y \in \mathcal V^{(n)}_\tau$.  As this holds for each $n$, it follows that $x \cup y \in \bigcap_n \mathcal V^{(n)}_\tau$.

In order to check that each $\mathcal V_\tau$ is proper, we begin by noting that $\mathcal V_\tau$ is an ideal which contains all finite variants of its members and, in particular, must contain the ideal $\Fin$.  That $\mathcal V_\tau \subsetneq 2^\omega$, for each $\tau$, follows by observing that, for each $x \in \mathcal V_\tau$, there is $y \in \mathcal V^{(0)}_\tau \subseteq \Coinf$ such that $x$ and $y$ are equal, mod finite.  In particular, for each $x \in \mathcal V_\tau$, we have $x \neq \omega$.

Finally, if $\I$ is an uncountable proper ideal on $\omega$, then $\I \nsubseteq \Fin$.  Therefore let $\tau \in 2^\omega$ be such that $\mathcal P (f((\tau)_0)) \subseteq \I$ and $\I = \mathcal U_{(\tau)_1}$.  We have $\mathcal V^{(0)}_\tau = \I$ and, in fact, each $\mathcal V^{(n)}_\tau = V^{(0)}_\tau$, since $\mathcal V^{(0)}_\tau$ is already a proper ideal containing $\mathcal P(f((\tau)_0))$.  This completes the claim and proof.
\end{proof}

As discussed in the Introduction, we have the following corollary.  The proof uses essentially the same argument that yields Theorem~\ref{T2} from Lemma~\ref{L1}.

\begin{refthmT7}
There is a wRK-complete proper uncountable $\mathbf \Pi^1_1$ ideal on $\omega$.\end{refthmT7}

As before, PD yields a corresponding result for the projective classes $\mathbf \Pi^1_{2n+1}$.

By slightly modifying the proof of Theorem~\ref{T1} one can also prove that there is a wRK-complete proper uncountable $\mathbf \Sigma^1_n$ ideal, for each $n \in \omega$.



\section{Projective equivalence relations}\label{S5}

We now give our proof of Hjorth's theorem on co-analytic equivalence relations, as well as our generalization of it under PD.

\begin{refthmT5}[Hjorth, 1996]
There is a universal $\Pi^1_1$ set for $\mathbf \Pi^1_1$ equivalence relations, i.e., there is a (lightface) $\Pi^1_1$ set $\mathcal E \subseteq 2^\omega \times 2^\omega \times 2^\omega$ such that the cross-sections $\mathcal E_\tau$ are exactly the $\mathbf \Pi^1_1$ equivalence relations on $2^\omega$.
\end{refthmT5}

\begin{refthmT4}[PD]
There is a universal set $\Pi^1_{2n+1}$ set for $\mathbf \Pi^1_{2n+1}$ equivalence relations on $2^\omega$.
\end{refthmT4}

\begin{proof}[Proof of Theorems~\ref{T5} and~\ref{T4}]
Let $\mathcal V \subseteq (2^\omega)^4$ and $\mathcal W \subseteq (2^\omega)^5$ be part of a good universal system for $\mathbf \Pi^1_{2n+1}$.  In particular, $\mathcal V$ and $\mathcal W$ are both $\Pi^1_{2n+1}$.  Let $\varphi : \mathcal V \rightarrow \mathsf{ORD}$ be a $\Pi^1_{2n+1}$-rank.  (For $n\geq 1$, the existence of such a $\varphi$ follows from PD.  If $n=0$, then such a $\varphi$ exists under ZFC.)  Define $\mathcal V^* \subseteq (2^\omega)^4$ by $(\alpha , \tau , x , y) \in \mathcal V^*$ iff $x=y$ or
\begin{enumerate}
\item\label{it5:1} $(x,y) , (y,x) \in \mathcal V_{\alpha, \tau}$,
\item\label{it5:2} $(\forall z) (((x,y),(y,z) \in \mathcal V_{\alpha , \tau} \ \& \ (\alpha , \tau , y , z) \leq_{\varphi} (\alpha , \tau , x , y)) \implies (x,z)\in \mathcal V_{\alpha , \tau})$ and
\item\label{it5:3} $(\forall z) (((z,x),(x,y) \in \mathcal V_{\alpha , \tau} \ \& \ (\alpha , \tau , z,x) \leq_{\varphi} (\alpha , \tau , x , y)) \implies (z,y) \in \mathcal V_{\alpha , \tau})$.
\end{enumerate} 
Note that, by the definability properties of $\Pi^1_{2n+1}$ ranks, $\mathcal V^*$ is itself $\Pi^1_{2n+1}$.

We now require a recursive function $f: 2^\omega \rightarrow 2^\omega$ such that $\mathcal V^*_\alpha = \mathcal V_{f(\alpha)}$, for each $\alpha \in 2^\omega$.  To produce such an $f$, let $\varepsilon \in 2^\omega$ be recursive and such that $\mathcal V^* = \mathcal W_\varepsilon$.  Thus, for each $(\alpha,\tau,x,y) \in (2^\omega)^4$,
\[
(\alpha,\tau,x,y) \in \mathcal V^* \iff (\varepsilon,\alpha,\tau,x,y) \in \mathcal W \iff ((s( \varepsilon,\alpha),\tau,x,y) \in V,
\]
where $s : (2^\omega)^2 \rightarrow 2^\omega$ is as in the definition of a good universal system.  Define $f$ by $f(\alpha) = s(\varepsilon, \alpha)$.

Let $\mathcal U \subseteq (2^\omega)^3$ be any universal $\Pi^1_{2n+1}$ set for $\mathbf \Pi^1_{2n+1}$.  The set $\mathcal U \cup (\{ (\tau , x,x) : \tau ,x \in 2^\omega\})$ is also $\Pi^1_{2n+1}$ and hence, by the definition of a good universal system, there is a recursive $\alpha_0 \in 2^\omega$ such that $\mathcal U \cup (\{ (\tau , x,x) : \tau , x \in 2^\omega\}) = \mathcal V_{\alpha_0}$.  Now define $\alpha_n = f^n (\alpha_0)$, i.e., $\alpha_{n+1} = f(\alpha_n)$, for each $n$.  Since $f$ is recursive, the map $n \mapsto \alpha_n$ is as well.  Thus, the set $\mathcal A = \bigcap_n \mathcal V_{\alpha_n}$ is $\Pi^1_{2n+1}$.

In the first place, each $\mathbf \Pi^1_{2n+1}$ equivalence relation on $2^\omega$ appears as a cross-section $\mathcal A_\tau$.  Indeed, if $E$ is a $\mathbf \Pi^1_{2n+1}$ equivalence relation, with $\tau$ such that $\mathcal U_\tau = E$, then each $(x,y) \in \mathcal V_{\alpha_0 , \tau}$ already satisfies conditions \ref{it5:1} - \ref{it5:3} above and hence
\[
\mathcal V_{\alpha_0 , \tau} = \mathcal V^*_{\alpha_0 , \tau} = \mathcal V_{\alpha_1 , \tau} = \mathcal V^*_{\alpha_1 , \tau} = \mathcal V_{\alpha_2 , \tau} = \ldots.
\]
It follows that $\mathcal A_\tau = \mathcal V_{\alpha_0 , \tau} = \mathcal U_\tau$.

The proof will be complete when we have verified that each $\mathcal A_\tau$ is an equivalence relation.  It is a consequence of the choice of $\alpha_0$ and the definition of $\mathcal V^*$ that $(x,x) \in \mathcal A_\tau$, for each $\tau,x \in 2^\omega$.  Also, condition \ref{it5:1} above guarantees that $(x,y) \in \mathcal A_\tau$ iff $(y,x) \in \mathcal A_\tau$.  To verify transitivity, suppose that $(x,y) , (y,z) \in \mathcal A_\tau = \bigcap_n \mathcal V_{\alpha_n} = \bigcap_n \mathcal V^*_{\alpha_n}$.  Fix $n$ and assume that $(\alpha_n , \tau , x,y) \leq_{\varphi} (\alpha_n , \tau , y,z)$.  Since $(y,z) \in \mathcal V^*_{\alpha_n , \tau}$, condition \ref{it5:3} above guarantees that $(x,z) \in \mathcal V_{\alpha_n , \tau}$.  If instead $(\alpha_n , \tau , y,z) \leq_{\varphi} (\alpha_n , \tau , x,y)$, then analogous reasoning, using condition \ref{it5:2}, shows that $(x,z) \in \mathcal V_{\alpha_n , \tau}$.  In either case, since $n$ was arbitrary, it follows that $(x,z) \in \bigcap \mathcal V_{\alpha_n , \tau} = \mathcal A_\tau$.  We conclude each $\mathcal A_\tau$ is an equivalence relation.  This completes the proof.
\end{proof}



\end{document}